\documentclass[a4paper, 12pt]{amsart}

\usepackage{amssymb}
\usepackage{amsthm}
\usepackage{amsmath}
\usepackage[mathscr]{eucal}
\usepackage{mathrsfs}

\usepackage{comment}

\usepackage{enumitem}

\theoremstyle{plain}
\newtheorem{thm}{Theorem}[section]		
\newtheorem{prop}[thm]{Proposition}
\newtheorem{cor}[thm]{Corollary}
\newtheorem{lem}[thm]{Lemma}

\theoremstyle{definition}
\newtheorem{df}{Definition}[section]
\newtheorem{exam}{Example}[section]

\theoremstyle{remark}
\newtheorem{rmk}{Remark}[section]
\newtheorem*{ac}{Acknowledgements}

\newcommand{\nn}{\mathbb{N}}
\newcommand{\zz}{\mathbb{Z}}
\newcommand{\qq}{\mathbb{Q}}
\newcommand{\rr}{\mathbb{R}}

\DeclareMathOperator{\card}{card}

\DeclareMathOperator{\pc}{\mathscr{P}}
\DeclareMathOperator{\nai}{INT}

\DeclareMathOperator{\seq}{Seq}
\DeclareMathOperator{\cycl}{Cycl}
\DeclareMathOperator{\conv}{\mathcal{CC}}
\DeclareMathOperator{\cm}{Comp}

\newcommand{\met}[1]{\mathrm{Met}(#1)}

\newcommand{\metdis}{\mathcal{D}}

\newcommand{\yopset}{\mathrm{PM}}

\newcommand{\yolset}[1]{\mathrm{L}_{#1}}
\newcommand{\yoeset}[1]{\mathrm{E}_{#1}}

\newcommand{\yodset}[1]{B(#1)}
\newcommand{\yokset}[1]{K(#1)}

\makeatletter
\@addtoreset{equation}{section}

\makeatother

\begin{document}

\title[An interpolation of metrics and spaces of metrics]
{An Interpolation of Metrics and Spaces of Metrics}

\author[Yoshito Ishiki]
{Yoshito Ishiki}

\address[Yoshito Ishiki]{
Department of Mathematical Sciences
\endgraf
Tokyo Metropolitan University
\endgraf
Minami-osawa Hachioji Tokyo 192-0397
\endgraf
Japan}

\email{ishiki-yoshito@tmu.ac.jp}

\date{\today}

\begin{abstract}
As a generalization of Hausdorff's extension theorem of metrics, we prove  an interpolation  theorem of a family of  metrics defined on closed subsets of metrizable spaces. As an application, we investigate typicality of subsets of moduli spaces of metrics. We observe that various  sets of all metrics with properties appearing  in metric geometry are dense intersections of countable open subsets  in spaces of metrics on metrizable spaces. For instance, our study is applicable to the set of all non-doubling metrics and the set of all non-uniformly disconnected metrics. 
\end{abstract}
\maketitle

\section{Introduction}
\subsection{Backgrounds}
In 1930, 
Felix 
Hausdorff 
\cite{Ha1930} 
proved the extension theorem stating  that for every metrizable space
 $X$, 
 for every closed subset 
 $A$ 
 of 
 $X$ 
 and for every 
metric 
$d$ 
on 
$A$
generating the same topology of 
$A$, 
there exists 
a metric 
$D$ 
on 
$X$ 
such that 
$D$ generates the same topology of 
$X$ and 
$D|_{A^2}=d$ 
(see Theorem 
\ref{thm:Hausdorff}). 
Motivated by this result, the author thought that Hausdorff's theorem suggested the study of spaces of metrics and tried to investigate their topological properties.

For a metrizable space 
$X$, 
we denote by 
$\met{X}$
 the set of all metrics on 
 $X$ 
generating 
the same topology of 
$X$. 
We define a function 
$\metdis_{X}: \met{X}\times \met{X}\to [0, \infty]$
 by 
 \[
\metdis_{X}(d, e)=\sup_{(x, y)\in X^2}|d(x, y)-e(x, y)|. 
 \]
The function 
$\metdis_{X}$
 is a metric on 
 $\met{X}$ 
 valued in 
 $[0, \infty]$. 
Throughout this paper, 
we consider that 
$\met{X}$
 is  always equipped with the topology generated by 
$\metdis_{X}$. 
To investigate 
topological properties of 
$\met{X}$, 
we
first  generalize the Hausdorff extension theorem to an interpolation theorem  of metrics with an approximation by 
$\metdis_{X}$
(Theorem
 \ref{thm:interpolate} or 
\ref{thm:ext241009}). 
As applications of it, 
for a
certain 
 property 
$\mathcal{P}$ on metric spaces, 
we prove that 
the set set 
$\{\, d\in \met{X}\mid\text{$(X, d)$ satisfies $\mathcal{P}$}\, \}$
is generic in 
$\met{X}$
(see Theorems \ref{thm:trans} and \ref{thm:loctrans}). 
This result 
can be considered as an analogue of
 Banach's famous result, which states 
  that 
in the continuous  function space on $[0, 1]$, 
a generic function is nowhere differentiable. 

Since a metric on 
$X$ 
determines the geometry on
 $X$, 
 studying the moduli space 
 $\met{X}$ 
 of metrics is equivalent to investigating the moduli of geometries that can be expanded   on the space
  $X$.

\subsection{Main results}

\subsubsection{Interpolation of metrics}
A family 
$\{S_{i}\}_{i\in I}$ 
of subsets of a topological space 
$X$ 
is said to be 
\emph{discrete} if for every 
$x\in X$
 there exists a neighborhood of
  $x$ intersecting at most single member of 
 $\{S_{i}\}_{i\in I}$. 

Our  first main results is the following interpolation theorem: 
\begin{thm}\label{thm:interpolate}
Let 
$X$ 
be a metrizable space, and let 
$\{A_{i}\}_{i\in I}$ 
be a discrete family of closed subsets of 
$X$. 
Then for every metric  
$d\in \met{X}$,   
and for every 
family 
$\{e_{i}\}_{i\in I}$ 
of metrics with 
$e_{i}\in \met{A_{i}}$, 
 there exists a metric 
 $m\in \met{X}$ 
 satisfying the following: 
\begin{enumerate}
\item for every 
$i\in I$ 
we have 
$m|_{A_{i}^2}=e_{i}$;
\item 
$\metdis_{X}(m, d)= \sup_{i\in I}\metdis_{A_{i}}(e_{A_{i}}, d|_{A_{i}^2})$. 
\end{enumerate}
Moreover, if 
$X$ 
is completely metrizable,  
and if each 
$e_{i}\in \met{A_{i}}$ 
is  a complete metric, then we can choose 
$m\in \met{X}$ 
as a complete one. 
\end{thm}

\begin{rmk}
For every  metrizable space 
$X$, 
and for every closed subset 
$A$ 
of 
$X$, 
Nguyen Van Khue and Nguyen To Nhu 
\cite{NN1981} 
constructed  
a Lipschitz metric extensor from 
$(\met{A}, \metdis_{A})$ 
into 
$(\met{X}, \metdis_{X})$, 
and a monotone continuous 
metric  extensor from 
$\met{A}$ 
into 
$\met{X}$;
moreover, 
if 
$X$ 
is completely metrizable, 
then each of these metric extensors maps any complete metric in 
$\met{A}$ 
into  a complete metric in 
$\met{X}$.  
To obtain 
such metric extensors, they used the Dugundji extension theorem concerning locally convex topological linear spaces. 
For more information of simultaneous extensions of 
metrics, see  
\cite{MR1751152}, 
\cite{MR3712976}, 
 \cite{MR1455501}, 
\cite{MR2054807}, 
\cite{MR2761694},
\cite{MR2857823}, 
\cite{MR2957499}, 
and 
\cite{MR2857999}. 
It seems difficult to prove 
Theorem 
\ref{thm:interpolate}
using these extension result because 
they do not preserve a given metric 
$d$. 
\end{rmk}

A central idea of the proof of Theorem \ref{thm:interpolate}
is a correspondence between a metric on a metrizable space and a topological embedding from a metrizable space into 
a Banach space. 
A metric 
$d$ 
 on a metrizable space 
 $X$ 
 induces a topological embedding 
from 
$X$ 
into a Banach space such as
 the Kuratowski embedding  
 (see Theorem \ref{thm:kur}). 
Conversely, a topological embedding 
$F$ 
from a metrizable space 
$X$ 
into a Banach space 
$V$ 
with norm 
$\|\cdot\|_V$ 
induces a metric 
$m\in \met{X}$ 
on 
$X$ 
defined by 
$m(x, y)=\|F(x)-F(y)\|_{V}$. 
In the proof of Theorem 
\ref{thm:interpolate}, 
we utilize this correspondence to translate
 the statement of Theorem 
 \ref{thm:interpolate} into an approximation  problem on 
 topological embeddings into a Banach space. 
We then  resolve such a problem by using  the Michael continuous selection theorem 
(see Theorem \ref{thm:Michael}), 
and by using a similar method to 
Kuratowski 
\cite{Ku1938} 
(see also 
\cite{Hausdorff1938}
 and
  \cite{MR0049543})
of converting a continuous function
 into a topological embedding by extending a codomain.

\subsubsection{Moduli spaces of metrics}
Theorem \ref{thm:interpolate} enables 
us to investigate  dense 
$G_{\delta}$ 
subsets in the topology of the space 
$(\met{X}, \metdis_{X})$ 
for a metrizable space 
$X$. 
To  describe our second result precisely, 
we define  a class of geometric properties that unifies  various properties appearing  
 in metric geometry. Due to this purpose, 
 the definition may seem somewhat complicated at first glance.

Let  
$\mathcal{P}^*(\nn)$ 
be  the set of all non-empty subsets of  
$\nn$. 
For a  topological space 
$T$, 
we denote by 
$\mathcal{F}(T)$ 
the set of all closed subsets of 
$T$. 
For a subset 
$W\in \mathcal{P}^*(\nn)$, 
and for a  set 
$S$, 
we denote by 
$\seq(W, S)$ 
the set of all finite injective sequences 
$\{a_i\}_{i=1}^n$ 
in  
$S$ with 
$n\in W$.

\begin{df}\label{def:transp}
Let 
$Q$ 
be an at most countable set, 
$P$ 
a topological space. 
Let 
$F: Q\to \mathcal{F}(P)$ 
and 
$G: Q\to \mathcal{P}^*(\nn)$ 
be maps.
Let 
$Z$ 
be a set.  
Let 
$\phi$ 
be a correspondence 
assigning a pair 
$(q, X)$  of 
$q\in Q$ 
and a metrizable space 
$X$ 
to a map 
$\phi^{q, X}:\seq(G(q), X)\times Z\times \met{X} \to P$. 
We say that 
a sextuple 
$(Q, P, F, G, Z, \phi)$ 
is 
a \emph{transmissible paremeter}
if 
for every metrizable space 
$X$, 
for every $q\in Q$, 
and
 for every 
 $z\in Z$  
 the following are satisfied:
\begin{enumerate}[label=\textup{(TP\arabic*)}]

	\item\label{item:tp1}
	for every 
	$a\in \seq(G(q), X)$, 
	and for every 
	$z\in Z$, 
	the map 
	$\phi^{q, X}(a, z):\met{X}\to P$ 
	defined by 
	$\phi^{q, X}(a, z)(d)=\phi^{q, X}(a, z, d)$ 
	is continuous;

	\item\label{item:tp2}
	 for every  
	 $d\in \met{X}$, 
	if  $S$ is a subset of $X$ and 
	$a\in \seq(G(q), S)$, 
	 then we have 
	$\phi^{q, X}(a, z, d)=\phi^{q, S}(a, z, d|_{S^2})$. 
	\end{enumerate}
\end{df}
The diameter 
$\phi^{q, X}(a, z, X)=\delta_{d}(\{a_{i}\}_{i=1}^{n})$ is a typical example 
of  a map appearing in  
the condition \ref{item:tp1}.

We introduce a property determined by a transmissible parameter. 
\begin{df}
Let 
$\mathfrak{G}=(Q, P, F, G, Z, \phi)$ 
be a transmissible parameter. 
We say that a metric space 
$(X, d)$ 
satisfies 
the 
\emph{$\mathfrak{G}$-transmissible  property} 
if 
there exists  
$q\in Q$ 
such that 
for every 
$a\in \seq(G(q), X)$ 
and for every 
$z\in Z$ 
we have 
$\phi^{q, X}(a, z, d)\in F(q)$. 
We say that 
$(X, d)$ 
satisfies the 
\emph{anti-$\mathfrak{G}$-transmissible property} 
if 
$(X, d)$  
satisfies the negation of the 
$\mathfrak{G}$-transmissible property; 
namely, 
for every 
$q\in Q$ 
there exist 
$a\in \seq(G(q), X)$ 
and 
$z\in Z$ 
with 
$\phi^{q, X}(a, z, d)\in X\setminus F(q)$. 
A property on metric spaces is  a 
\emph{transmissible property} (resp.~\emph{anti-transmissible property})
if it is equivalent to a 
$\mathfrak{G}$-transmissible property 
(resp.~anti-$\mathfrak{G}$-transmissible property) 
for some
transmissible parameter 
$\mathfrak{G}$. 
\end{df}

The class of transmissible properties contains various properties appeared in metric geometry. 
\begin{exam}\label{exam:prop}
The following properties on metric spaces are transmissible properties  
(see Section \ref{sec:trans}). 
\begin{enumerate}[label=\textup{(\arabic*)}]
\item\label{item:doubling1010}
 the doubling property; 

\item  
the uniform disconnectedness; 

\item 
satisfying the ultratriangle inequality; 

\item  
satisfying the Ptolemy inequality; 

\item\label{item:gcyc1010}
 the Gromov 
$\cycl_{m}(0)$  condition;

\item the Gromov hyperbolicity. 
\end{enumerate}
\end{exam}

To state  our second result, we need a more additional condition for  transmissible properties. 
\begin{df}
Let 
$\mathfrak{G}=(Q, P, F, G, Z, \phi)$ 
be a transmissible parameter. 
We say that 
$\mathfrak{G}$ 
is 
\emph{singular} 
if for each 
$q\in Q$ 
and for every 
$\epsilon\in (0, \infty)$
there exist 
$z\in Z$ 
and  a finite metrizable  space 
$(L, d_{L})$ 
such that 
\begin{enumerate}

	\item 
	$\delta_{d_{L}}(L)\le \epsilon$, 
	where 
	$\delta_{d_{L}}(L)$ 
	stands for the diameter of $L$; 
	
	\item 
	$\card(L)\in G(q)$, 
	where 
	$\card$ 
	stands for the cardinality; 
	
	\item 
	$\phi^{q, L}(L, z, d_{L})\in X\setminus F(q)$. 
\end{enumerate}
\end{df}
As we will  see in 
Lemma \ref{lem:sing}, 
the singularity of  a transmissible parameter
$\mathfrak{G}$  is 
equivalent  to the existence of 
a special countable metric subspace  that 
satisfies  the anti-$\mathfrak{G}$-transmissible property. 

Note that not all transmissible parameters are singular;
especially, 
the Gromov hyperbolicity 
does not have
 a singular transmissible  parameter (see Proposition \ref{prop:Gro}).

Due to Theorem 
\ref{thm:interpolate}, 
we obtain the second main result on dense 
$G_{\delta}$ 
subsets in spaces of metrics:
\begin{thm}\label{thm:trans}
Let 
$\mathfrak{G}$ 
be a singular transmissible parameter. 
For every non-discrete metrizable space $X$, 
the set of all 
$d\in \met{X}$ 
for which 
$(X, d)$ 
satisfies the
anti-$\mathfrak{G}$-transmissible property
 is dense 
 $G_{\delta}$ in $\met{X}$.  
\end{thm}

\begin{rmk}
Theorem 
\ref{thm:trans} 
holds true for the 
space 
$\cm(X)$ 
of all complete metrics in 
$\met{X}$ 
(see Theorems \ref{thm:cm1}). 
\end{rmk}

We can apply Theorem 
\ref{thm:trans} 
to the properties
\ref{item:doubling1010}--\ref{item:gcyc1010}
mentioned  in Example \ref{exam:prop}. 
Therefore we conclude
 that
  the set of all metrics not satisfying these properties 
are dense 
$G_{\delta}$ 
in spaces of metrics 
(see 
Theorem \ref{thm:dandud:241018} 
and Corollary \ref{cor:inqopen}). 
We also conclude that the set of all metrics with rich pseudo-cones is 
dense 
$G_{\delta}$ 
in spaces of metrics 
(see Theorem \ref{thm:richcones}).

Our third result is based on the fact that for 
every 
second-countable locally compact space 
$X$, 
the space 
$\met{X}$ 
is  Baire 
(see Lemma 
\ref{lem:Baire}). 
For a property 
$P$ 
on metric spaces, 
we say that a metric space 
$(X, d)$ 
satisfies the 
\emph{local $P$} 
if 
every non-empty open metric  subspace of 
$X$ 
satisfies the property 
$P$.

As a local version of Theorem 
\ref{thm:trans}, 
we obtain the following:
\begin{thm}\label{thm:loctrans}
Let 
$X$ 
be a second-countable, 
locally compact locally non-discrete space. 
Then for every singular transmissible parameter 
$\mathfrak{G}$, 
the set of all metrics 
$d\in \met{X}$ 
for which 
$(X, d)$ 
satisfies 
the local 
anti-$\mathfrak{G}$-transmissible property
 is dense 
 $G_{\delta}$ 
 in 
 $\met{X}$. 
\end{thm}
Recall 
 that 
all second-countable locally compact spaces are metrizable,  
which is a consequence of 
the Urysohn metrization theorem.

\subsubsection{Organization}
This paper is organized  as follows. 
In Section 
\ref{sec:pre}, 
we review 
the basic and  classical  theorems 
on topological spaces 
and metric spaces. 
Section \ref{sec:inter} is devoted to
the  proof of  Theorem 
\ref{thm:interpolate}. 
In Section 
\ref{sec:trans}, 
we prove Theorem 
\ref{thm:trans} 
and show that 
various properties in metric geometry are transmissible properties. 
 Section 
\ref{sec:loctrans}
contains the proof that 
$\met{X}$ is Baire when 
$X$ is second-countable and locally compact. 
We also show  Theorem 
\ref{thm:loctrans}.

\begin{ac}
The author would like to thank Professor Koichi Nagano for his advice and constant encouragement. 
This work was supported by JSPS KAKENHI Grant Number 18J21300, 
and partially supported by 
JSPS 
KAKENHI Grant Number 
JP24KJ0182. 
The author also would like to 
thank the referee
 for variable and helpful comments to 
improve the present paper. 
\end{ac}

\section{Preliminaries}\label{sec:pre}
In this paper, the symbol 
$\nn$ 
stands for the set of all positive integers.
 
Let 
$(X, d)$ 
be a metric space. 
Let 
$A$ 
be a subset of
 $X$.  
We denote by 
$\delta_{d}(A)$ 
the diameter of 
$A$. 
We denote by 
$B(x, r)$ 
(resp.~$U(x, r)$) 
the closed 
(resp.~open) ball centered at 
$x$ 
with radius 
$r$. 
\subsection{The Hausdorff extension theorem}
The following celebrated theorem was first proven by Hausdorff 
\cite{Ha1930}
(cf. \cite{Hausdorff1938}, 
\cite{MR0024609}, 
\cite{MR0230285}, 
and 
\cite{MR321026}). 
The latter part on complete
metrics 
was added by 
Bacon 
\cite{MR0049543}. 
\begin{thm}\label{thm:Hausdorff}
Let 
$X$ be a metrizable space,  
and let 
$A$ 
be a closed subset of 
$X$. 
Then for every 
$d\in \met{A}$ 
there exists 
$D\in \met{X}$ 
with 
$D|_{A^2}=d$. 
Moreover, 
if 
$X$ is completely metrizable,  
and if 
$d\in \met{A}$ 
is  a complete metric on 
$A$, 
then we can choose 
$D\in \met{X}$ 
as a complete metric on 
$X$. 
\end{thm}

\subsection{The Michael continuous selection theorem}
Let 
$V$ 
be a Banach space. 
We denote by 
$\conv(V)$ 
the set of all non-empty closed convex subsets  of 
$V$. 
For a topological space 
$X$ 
we say that a map 
$\phi\colon  X\to \conv(V)$ 
is 
\emph{lower semi-continuous} 
if for every open subset 
$O$ 
of 
$V$ 
the set 
$\{\, x\in X\mid \phi(x)\cap O\neq \emptyset\, \}$ 
is open in $X$. 

The following  theorem is known as one of the Michael continuous selection theorems proven in 
\cite{MR77107}.

\begin{thm}\label{thm:Michael}
Let
 $X$ 
 be a paracompact space, 
and 
$A$ 
 a closed subsets of 
 $X$, 
$V$  a Banach space, 
and 
$\phi\colon X\to \conv(V)$ 
be a lower semi-continuous map. 
If  a continuous map 
$f\colon A\to B$ 
satisfies 
$f(x)\in \phi(x)$ 
for all 
$x\in A$, 
then there exists a continuous map 
$F\colon X\to V$ with 
$F|_A=f$ 
such that 
for every 
$x\in X$ 
we have 
$F(x)\in \phi(x)$. 
\end{thm}

By using  linear structure, 
we have the following:
\begin{prop}\label{prop:Hiso}
Let $V$ be a Banach space,  
and 
$x, y\in V$. 
Then for every 
$r\in (0, \infty)$ 
we have 
$\mathcal{H}(B(x, r), B(y, r))= \|x-y\|_V$, 
where 
$\mathcal{H}$ 
is the Hausdorff distance induced from the norm 
$\|\cdot\|_V$ 
of  
$V$. 
\end{prop}

\begin{cor}\label{cor:lsc}
Let 
$X$ 
be a topological space,  
$V$ 
be a Banach space, 
and let
$H\colon X\to V$ 
be a continuous map and 
$r\in (0, \infty)$. 
Then a map 
$\phi\colon X\to \conv(V)$ 
defined by 
$\phi(x)=B(H(x), r)$ 
is lower semi-continuous. 
\end{cor}
\begin{proof}
For every open subset 
$O$ 
of 
$V$,  
and for every point 
$a\in X$ 
with 
$\phi(a)\cap O\neq \emptyset$, 
choose  
$u\in \phi(a)\cap O$ 
and  
$l\in (0, \infty)$ 
with 
$U(u, l)\subset O$. 
By Proposition \ref{prop:Hiso}, 
we can take 
$\delta\in (0, \infty)$ 
such that for every 
 $x\in U(a, \delta)$ 
 we have 
 $\mathcal{H}(\phi(x), \phi(a))=\|H(x)-H(a)\|_V<l$. 
Then we have 
$\phi(x)\cap U(u, l)\neq \emptyset$, 
and hence  
$\phi(x)\cap O\neq \emptyset$. 
Therefore the set 
$\{\, x\in X\mid \phi(x)\cap O\neq \emptyset\, \}$ 
is open in $X$. 
\end{proof}

The following theorem is 
known as the Stone theorem proven in  
\cite{MR0026802}. 
\begin{thm}\label{thm:Stone}
All metrizable spaces are paracompact. 
\end{thm}
By this theorem, 
we can apply Theorem 
\ref{thm:Michael} 
to all metrizable space.

\subsection{The Kuratowski embedding theorem}
For a metric space 
$(X, d)$, 
we denote by 
$C_b(X)$ 
the Banach space of all bounded continuous functions on
 $X$ 
equipped with the supremum norm. 
For 
$x\in X$, 
we denote by 
$d_{x}$ 
the function from 
$X$ 
to 
$\rr$ 
defined by 
$d_x(p)=d(x, p)$. 
The following theorem, 
which states that every metric space
is isometrically embeddable into Banach spaces,  
is known as the Kuratowski embedding theorem. 
The proof can be seen in, for example,  
\cite[Theorem 4.3.14]{MR1039321}. 

\begin{thm}\label{thm:kur}
Let 
$(X, d)$ 
be a metric space. 
Take 
$o\in X$. 
Then 
the  map 
$K: X\to C_{b}(X)$ 
defined  by 
$K(x)=d_{x}-d_{o}$ 
is an isometric embedding. 
Moreover, 
if 
$(X, d)$ 
is bounded, 
the map 
$L\colon X\to C_{b}(X)$ 
defined by 
$L(x)=d_{x}$ 
is an isometric embedding. 
\end{thm}

\subsection{Baire spaces}
A topological space
 $X$ is said to be  
 \emph{Baire} if the intersections of 
countable dense open subsets of 
$X$ 
are dense in 
$X$.

The following  is 
 known as the Baire category theorem. 
\begin{thm}
Every completely metrizable space is a Baire space. 
\end{thm}

Since 
$G_{\delta}$ 
subset of completely metrizable space is 
completely metrizable 
(see, e.g.  \cite[Theorem 24.12]{MR2048350}), 
we obtain the following:
\begin{lem}\label{lem:gdelta}
Every 
$G_{\delta}$ 
subset of a completely metrizable space is a Baire space. 
\end{lem}


\section{An interpolation Theorem of metrics}\label{sec:inter}

In this section, we prove Theorem 
\ref{thm:interpolate}.
\subsection{Amalgamation lemmas}
Let 
$X$ 
and 
$Y$ 
be sets,  
and let
 $\tau\colon  X\to Y$ 
 be a bijection. 
For a metric 
$d$ 
 on 
 $Y$,  
we denote by 
$\tau^{*}d$ 
the metric on 
$X$ 
defined by 
$(\tau^{*}d)(x, y)=d(\tau(x), \tau(y))$. 
In this case, 
 the map 
$\tau$ 
is an isometry between 
$(X, \tau^{*}d)$ 
and 
$(Y, d)$.

The following proposition can be considered as a specific case 
of the realization of  
the  Gromov--Hausdorff distance of two metric spaces. 
\begin{prop}\label{prop:GH}
Let 
$X$ 
be a metrizable space. 
For  
$r\in (0, \infty)$, 
let 
$d, e\in \met{X}$ 
with 
$\metdis_{X}(d, e)\le r$. 
Put 
$X_{0}=X$, 
and 
let 
$X_{1}$ 
be a set with 
$\card(X_1)=\card(X_0)$ 
and 
$X_{0}\cap X_{1}=\emptyset$. 
Take 
a bijection 
$\tau :X_{0}\to X_{1}$. 
Then there exists a metric 
$h\in \met{X_{0}\sqcup X_{1}}$ 
such that 
\begin{enumerate}
	\item $h|_{(X_{0})^2}=d$;
	\item $h|_{(X_{1})^2}=(\tau^{-1})^{*}e$; 
	\item for every 
	$x\in X_{0}$ 
	we have 
	$h(x, \tau(x))= r/2$. 
\end{enumerate}
\end{prop}
\begin{proof}
We define a symmetric function 
$h: (X_{0}\sqcup X_{1})^2\to [0, \infty)$ 
by
\begin{align*}
	h(x, y)=
		\begin{cases}
		d(x, y) & \text{if $x, y\in X_0$;}\\
		e(x, y) & \text{if $x, y\in X_1$;}\\
		 \inf_{a\in X_0}(d(x, a)+r/2+e(\tau(a), y)) & \text{if $(x, y)\in X_0\times X_1$. }
		\end{cases}
\end{align*}
Then $h$ is a metric as desired. 
For more details, we refer the readers to  \cite[Theorem 7.3.25]{MR1835418}. 
\end{proof}

The following   proposition
is  known as  an amalgamation of metrics.  
\begin{prop}\label{prop:amal1}
Let 
$(X, d_{X})$ 
and 
$(Y, d_{Y})$ 
be metric spaces, 
and let 
$Z=X\cap Y$. 
If 
$Z\neq \emptyset$ 
and 
$d_{X}|_{Z^2}=d_{Y}|_{Z^2}$, 
then there exists a metric
 $h$ on 
$X\cup Y$ 
such that 
\begin{enumerate}
	\item $h|_{X^2}=d_{X}$; 
	\item $h|_{Y^2}=d_{Y}$. 
\end{enumerate}
\end{prop}
\begin{proof}
We define a symmetric function 
$h: (X\cup Y)^2\to [0, \infty)$ 
by
\begin{align*}
	h(x, y)=
		\begin{cases}
		d_X(x, y) & \text{if $x, y\in X$;}\\
		d_Y(x, y) & \text{if $x, y\in Y$;}\\
		\inf_{z\in Z}(d_X(x, z)+d_Y(z, y)) & \text{if $(x, y)\in X\times Y$. }
		\end{cases}
\end{align*}
Since 
$d_{X}|_{Z^2}=d_{Y}|_{Z^2}$, 
the function 
$h$
 is well-defined, and it becomes a desired metric. 
 The detailed proof can be seen in, for example, 
\cite{MR1918193}
 (cf. \cite{Ury1927} and \cite{MR105082}). 
\end{proof}

 For a mutually disjoint family 
$\{T_{i}\}_{i\in I}$ 
of  topological spaces, 
we consider that  the space 
$\coprod_{i\in I}T_{i}$ 
is always equipped
 with the direct sum topology.

\begin{prop}\label{prop:key}
Let 
$X$ 
be a metrizable space, 
and 
$\{A_{i}\}_{i\in I}$ 
be a discrete family of closed subsets of 
$X$. 
Take 
$d\in \met{X}$,  
and let 
$\{e_{i}\}_{i\in I}$ 
be a family of metrics with 
$e_{i} \in \met{A_{i}}$. 
Put  
$\eta=\sup_{i\in I}\metdis_{A_{i}}(e_{A_{i}}, d|_{(A_{i})^2})$
and assume that 
$\eta<\infty$. 
Let 
$\{B_{i}\}_{i\in I}$ 
be a mutually disjoint family of sets 
such that for all 
$i\in I$ 
we have 
$X\cap B_{i}=\emptyset$. 
Take 
$\tau\colon  \coprod_{i\in I}A_{i}\to \coprod_{i\in I}B_{i}$ 
be a bijection such that 
for each 
$i\in I$ 
the map 
$\tau_{i}=\tau|_{A_{i}}$ 
is a bijection between 
$A_{i}$ 
and 
$B_{i}$. 
Then 
there exists a metric 
$h$ 
on 
$X\sqcup\coprod_{i\in I}B_{i}$ 
such that
\begin{enumerate}
	\item for every 
	$i\in I$ 
	we have 
	$h|_{(B_{i})^2}=(\tau_i^{-1})^{*}e_{i}$; 
	\item $h|_{X^2}=d$; 
	\item for every 
	$x\in \coprod_{i\in I}A_{i}$ 
	we have $h(x, \tau(x))= \eta/2$. 
\end{enumerate}
\end{prop}
\begin{proof}
for every  
$i\in I$, 
Proposition \ref{prop:GH}
enables us to 
take   a metric 
$l_{i}\in \met{A_{i}\sqcup B_{i}}$ 
such that 
\begin{enumerate}[label=\textup{(a\arabic*)}]
	\item $l_{i}|_{A_{i}^2}=d|_{A_{i}^2}$;
	\item $l_{i}|_{B_{i}^2}=(\tau_i^{-1})^{*}e_{i}$; 
	\item for all $x\in A_i$ 
	we have $l_{i}(x, \tau(x))= \eta/2$. 
\end{enumerate}
From now on, 
 we construct a metric mentioned in the proposition 
by transfinite recursion.
We may assume that 
the index set $I$ is a cardinal 
$\kappa$. 
For each 
$\gamma<\theta$, 
we put 
$Z_{\gamma}=X\sqcup\coprod_{\alpha<\gamma}B_{\alpha}$. 

Fix 
$\theta\le \kappa$, 
and 
assume that 
we already constructed a metric 
$\{h_{\alpha}\}_{\alpha<\theta}$
 such that 
\begin{enumerate}[label=\textup{(b\arabic*)}]

\item\label{item:amal241011}
If $\alpha<\beta<\theta$, 
then we have 
$h_{\beta}|Z_{\alpha}=h_{\alpha}$. 

\item\label{item:restrict241011}
If 
$\alpha<\beta<\theta$, 
then 
we have 
$h_{\beta}|_{(A_{\alpha}\cup B_{\alpha})^{2}}=l_{\alpha}$. 
\end{enumerate}

Under this assumptions, 
we shall make 
$h_{\theta}$. 
We divide the 
construction into 
two cases. 

If $\theta=\gamma+1$ for some $\gamma$, 
then 
we 
apply
Proposition 
\ref{prop:amal1}
 to 
$h_{\gamma}$ 
and 
$l_{\gamma}$, 
and 
we then  obtain 
$h_{\theta}$
satisfying 
\ref{item:amal241011}
and 
\ref{item:restrict241011}
for $\theta+1$. 

If 
$\theta$ is a limit cardinal, 
then 
we define 
a metric 
$h_{\theta}$  on 
$Z_{\theta}$ by 
$h_{\theta}(x, y)=h_{\alpha}(x, y)$, 
where $\alpha<\theta$ satisfies 
that 
$x, y\in Z_{\alpha}$. 
The existence of $\alpha$ is guaranteed by 
$Z_{\theta}=\bigcup_{\alpha<\theta}Z_{\alpha}$. 
By \ref{item:amal241011}, 
the metric 
$h_{\theta}$ is well-defined. 
Similarly, 
we see that 
$h_{\theta}$
satisfies 
\ref{item:amal241011}
and 
\ref{item:restrict241011}
for $\theta+1$.

Put 
$h=h_{\kappa}$. 
From the assumptions 
 \ref{item:amal241011}
 and 
  \ref{item:restrict241011} for $\kappa$, 
  and from 
$Z_{\kappa}=X\sqcup \coprod_{i\in I}B_{i}$, 
it follows that 
$h$ 
is as required. 
\end{proof}

\subsection{Proof of Theorem \ref{thm:interpolate}}

To show Theorem \ref{thm:interpolate}, 
we first prove the case 
where the index set is a singleton. 
\begin{thm}\label{thm:ext241009}
Let
 $X$ 
 be a metrizable space,  
and let 
$A$ 
be a closed subset of
 $X$.
Then  
for every  
$d\in \met{X}$, 
and for every  
$e\in \met{A}$, 
there exists  a metric 
$m\in \met{X}$ 
such that:
\begin{enumerate}
	\item $m|_{A^2}=e$; 
	\item $\metdis_{X}(m, d)= \metdis_{A}(e, d|_{A^2})$. 
\end{enumerate}
Moreover, 
if 
$X$ 
is completely metrizable,  
and if 
$e\in \met{A}$ 
is  a complete metric, 
then we can choose
 $m\in \met{X}$ as a complete metric. 
\end{thm}
\begin{proof}
Put 
$\eta=\metdis_{A}(e, d|_{A^2})$. 
If 
$\eta=\infty$, 
then
the theorem  follows 
from 
the Hausdorff extension theorem 
\ref{thm:Hausdorff}.

We may assume 
$\eta<\infty$.
Let 
$B$ 
be a set such that 
$X\cap B=\emptyset$
and $\card(B)=\card(A)$, 
and 
let 
$\tau\colon  A\to B$ 
be a bijection. 
Put 
$Z=X\sqcup B$. 
By Proposition 
\ref{prop:key},  
we find a metric
 $h$ 
 on 
 $Z$ 
 such that
\begin{enumerate}[label=\textup{(H\arabic*)}]
	\item
	  we have $h|_{B^{2}}=(\tau^{-1})^{*}e$; 
	
	\item 
	$h|_{X^{2}}=d$; 
	
	\item\label{item:h:1016}
	for every 
	$x\in A$ 
	we have 
	$h(x, \tau(x))= \eta/2$. 
\end{enumerate}

We can take an isometric embedding 
$H\colon Z\to Y$ 
from 
$(Z, h)$ 
into a Banach space 
$(Y, \|\cdot\|_{Y})$ 
(see e.g.,  the Kuratowski embedding theorem \ref{thm:kur}). 
Define a map 
$\phi\colon  Z\to \conv(Y)$ 
by 
$\phi(x)=B(H(x), \eta/2)$. 
By Corollary 
\ref{cor:lsc}, 
the map 
$\phi$
 is lower semi-continuous. 
We define a map 
$f\colon  A\to Y$ 
by 
$f(x)=H(\tau(x))$. 
Then 
$f$ 
is continuous. 
By the property  
\ref{item:h:1016}
 of 
 $h$,  
for every 
$x\in A$ 
we have 
$f(x)\in \phi(x)$. 
Due to the Stone theorem
 \ref{thm:Stone},  
the space 
$X$ 
is paracompact. 
Thus
we can  apply  the Michael selection  theorem \ref{thm:Michael} 
to the map 
$f$, 
and hence
we obtain a continuous map
$F\colon X\to Y$ 
such that 
$F|_{A}=f$ 
and for every 
$x\in A$ 
we have 
$F(x)\in \phi(x)$. 
Note that 
$F(x)\in \phi(x)$ 
means that 
\[
\|F(x)-H(x)\|_Y\le \eta/2.
\]

Define a map 
$l\colon X\times X\to [0, \infty)$ 
 by 
$l(x, y)=\min\{e(x, y), \eta/2\}$. 
Note that 
$l\in \met{X}$. 
We consider that the product Banach space 
$Y\times C_{b}(X)$ 
is equipped with 
the max norm defined by 
$\|(x, y)\|=\max\{\|x\|_{Y}, \|y\|_{C_{b}(X)}\}$. 

Define a map 
$E\colon X\to Y\times C_{b}(X)$ 
by 
	\[
	E(x)=(F(x), l_x), 
	\]
where 
$l_{x}$ 
is a bounded function on 
$X$  
defined by 
$l_{x}(p)=l(x, p)$. 
By the Kuratowski embedding theorem 
\ref{thm:kur}, 
the map 
$L\colon  X\to C_b(X)$ 
defined by 
$L(x)=l_{x}$ 
is an isometric embedding.   
Therefore 
$E$ 
is a topological embedding. 
We also define a map 
$K\colon  X\to Y\times C_{b}(X)$ 
by 
	\[
	K(x)=(H(x), 0). 
	\]
Then, 
by the definition of the norm of 
$Y\times C_{b}(X)$,  
the map 
$K$ 
from 
$(X, d)$ 
to 
$(Y\times C_{b}(X), \|\cdot\|)$ 
is an isometry. 
Since for every 
$x\in X$ 
we have 
$\|F(x)-H(x)\|_{Y}\le \eta/2$ 
and  
$\|l_x\|_{C_{b}(X)}\le \eta/2$, 
we obtain 
\begin{align}
	\|E(x)-K(x)\|=\max\{\|F(x)-H(x)\|_{Y}, \|l_x\|_{C_{b}(X)}\}\le \eta/2.\label{al:keyinq}
\end{align}

Define a function 
$m:X^2\to [0, \infty)$ 
by 
$m(x, y)=\|E(x)-E(y)\|$, 
then 
$m$ 
is a metric on 
$X$. 
Since 
$E$
 is a topological embedding, 
we see that 
$m\in \met{X}$. 
For every 
pair  
$x, y\in A$, 
we have 
$\|F(x)-F(y)\|_{Y}=e(x, y)$ 
and 
	\[
	\|l_{x}-l_{y}\|_{C_{b}(X)}=l(x, y)\le e(x, y);
	\] 
thus we obtain 
	\[
	\|E(x)-E(y)\|=\max\{\|F(x)-F(y)\|_{Y}, \|l_{x}-l_{y}\|_{C_{b}(X)}\}=e(x, y), 
	\]
and hence 
$m|_{A^2}=e$. 
Moreover,  
we have 
$\eta\le \metdis_{X}(m, d)$. 
We also obtain the opposite inequality 
$\metdis_X(m, d)\le \eta$;
indeed, 
the inequality 
\eqref{al:keyinq}
shows that 
for every pair 
$x, y\in X$, we have 
\begin{align*}
	&|m(x, y)-d(x, y)|=\Bigl|\|E(x)-E(y)\|-\|K(x)-K(y)\|\Bigr|
	\\
	&\le \|E(y)-K(y)\|+\|E(x)-K(x)\|
	\le \eta/2 +\eta/2=\eta. 
\end{align*}
Therefore we conclude that 
$\metdis_{X}(m, d)=\eta$. 
This completes the proof of the  former part of Theorem 
\ref{thm:interpolate}. 

By the latter part of the Hausdorff theorem \ref{thm:Hausdorff}, 
we can choose $l$ as a complete metric. 
Then $m$ become a complete metric. 
This shows  the  latter part of Theorem \ref{thm:interpolate}. 
\end{proof}

\begin{proof}[Proof of Theorem \ref{thm:interpolate}]
Let 
$X$ 
be a metrizable space, and let 
$\{A_{i}\}_{i\in I}$ 
be a discrete family of closed subsets of 
$X$. 
Put $A=\bigcup_{i\in I}A_{i}$. 
Since $\{A_{i}\}_{i\in I}$ is a discrete family, 
the union set 
$A=\bigcup_{i\in I}A_{i}$ is 
also closed in 
$X$. 
We use the same notation of
Proposition 
\ref{prop:key}. 
We define 
$w=\tau^{*}h\in \met{A}$. 
note that 
$w|_{A_{i}^{2}}=e_{i}$ 
for every 
$i\in I$. 
Let us check that 
$\metdis_{A}(d|_{A^{2}}, w)=\eta$. 
In fact, 
for every 
pair 
$x, y\in A$, we have 
\begin{align*}
|d(x, y)-w(x, y)|&\le |h(x, y)-h(x, \tau(y))|+|h(x, \tau(y))-h(\tau(x), \tau(y))|\\
&\le h(y, \tau(y))+h(x, \tau(x))\le \eta/2+\eta/2=\eta. 
\end{align*}
Hence 
$\metdis_{A}(d|_{A^{2}}, w)=\eta$. 
Applying 
Theorem 
\ref{thm:ext241009}
to 
$d$ 
and 
$w$, 
we obtain 
$m\in \met{X}$ such that 
$\metdis_{X}(d, m)=\eta$ and 
$m|_{A^{2}}=w$.
This finishes the proof of Theorem \ref{thm:interpolate}. 
\end{proof}


\section{Transmissible  properties}\label{sec:trans}
In this section we discuss  transmissible properties, 
and prove Theorem \ref{thm:trans}. 
We also show that 
various properties in metric geometry are transmissible properties.  
\subsection{Proof of Theorem \ref{thm:trans}}
By the condition 
\ref{item:tp2} in Definition \ref{def:transp}, 
we obtain the following:
\begin{lem}\label{lem:here}
Let $\mathfrak{G}$ be a transmissible parameter. 
If a metric space 
$(X, d)$ 
satisfies the 
$\mathfrak{G}$-transmissible property, 
then so does every metric subspace of 
$(X, d)$. 
\end{lem}
By the virtue of Lemma \ref{lem:here}, 
we use the word ``transmissible''. 
\begin{cor}\label{cor:anti}
Let $\mathfrak{G}$ be a transmissible parameter, 
and let $(X, d)$ be a metric space. 
If there exists a metric subspace of 
$(X, d)$ 
satisfies the anti-$\mathfrak{G}$-transmissible property, 
then so does $(X, d)$. 
\end{cor}

Let $X$ be a metrizable space, 
and let 
$\mathfrak{G}=(Q, P, F, G, Z, \phi)$ 
be a transmissible  parameter. 
For $q\in Q$, 
for 
$a\in \seq(G(q), X)$ 
and for
$z\in Z$, 
we denote by 
$S( X, \mathfrak{G}, q, a, z)$ 
the set of all 
$d\in \met{X}$ 
such that 
$\phi^{q, X}(a, z, d)\in X\setminus F(q)$. 
We also denote by 
$S(X, \mathfrak{G})$ 
the set of all 
$d\in \met{X}$ 
such that 
$(X, d)$ 
satisfies the 
anti-$\mathfrak{G}$-transmissible property.

\begin{prop}\label{prop:open}
Let $X$ be a metrizable space, 
$\mathfrak{G}=(Q, P, F, G, Z, \phi)$ 
 a transmissible  parameter. 
Then 
for every 
$q\in Q$, 
for every 
$a\in \seq(G(q), X)$, 
and 
for every 
$z\in Z$, 
the set 
$S(X, \mathfrak{G}, q, a, z)$
is open in 
$\met{X}$. 
\end{prop}
\begin{proof}
Fix $q\in Q$, 
$a\in \seq(G(q), X)$ 
and 
$z\in Z$. 
Since the map 
$\phi^{q, X}(a, z)\colon \met{X}\to P$ 
is continuous
and since 
$X\setminus F(q)$
 is open in 
$P$, 
 the set 
 $S(X, \mathfrak{G}, q, a, z)$ 
 is open in 
 $\met{X}$. 
\end{proof}
\begin{cor}\label{cor:open}
Let $X$ be a metrizable space, 
$\mathfrak{G}=(Q, P, F, G, Z, \phi)$ 
 a transmissible  parameter. 
Then 
the set 
$S(X, \mathfrak{G})$ 
is 
$G_{\delta}$ 
in 
$\met{X}$.
Moreover, 
if the set 
$Q$ 
is finite, 
then 
$S(X, \mathfrak{G})$ 
is open in 
$\met{X}$. 
\end{cor}
\begin{proof}
By the definitions of 
$S(X, \mathfrak{G})$ 
and 
$S(X, \mathfrak{G}, q, a, z)$, 
we have 
\[
S(X, \mathfrak{G})
=
\bigcap_{q\in Q}
\bigcup_{a\in \seq(G(q), X)}
\bigcup_{z\in Z}
S(X, \mathfrak{G}, q, a, z). 
\]
This equality and  Proposition
 \ref{prop:open} prove
the lemma. 
\end{proof}

We say that a topological space is an
 \emph{$(\omega_{0}+1)$-space}
if it is homeomorphic to the one-point compactification of the countable 
discrete topological  space. 
Namely, it is  homeomorphic to the 
ordinal 
$\omega_{0}+1$. 
\begin{lem}\label{lem:sing}
Let 
$\mathfrak{G}$
 be a singular transmissible parameter. 
Then there exists an 
$(\omega_{0}+1)$-metric space 
with arbitrary small diameter 
satisfying the 
 anti-$\mathfrak{G}$-transmissible property. 
\end{lem}
\begin{proof}
Let 
$\mathfrak{G}=(Q, P, F, G, Z, \phi)$. 
Fix 
$\epsilon\in (0, \infty)$. 
By the singularity of 
$\mathfrak{G}$, 
there exists a sequence 
$\{(R_{i}, d_{i})\}_{i\in \nn}$ 
of finite metric spaces  such that for each 
$i\in \nn$ 
there exist
$q_{i}\in Q$ 
and 
$z_{i}\in Z$ 
satisfying
\begin{enumerate}[label=\textup{(R\arabic*)}]
	
	\item\label{item:ep1017}
	 $\delta_{d_{i}}(R_{i})\le \epsilon\cdot2^{-i}$; 
	
	\item\label{item:rin1016}
	 $\card(R_{i})\in G(q_{i})$; 
	 
	\item\label{item:rphi1016}
	$\phi^{q_{i}, R_{i}}(R_{i}, z_{i}, d_{i})\in X\setminus F(q_{i})$. 
\end{enumerate}
Put
	\[
	L=\{\infty\}\sqcup \coprod_{i\in \nn}R_{i}, 
	\]
and define a metric 
$d_{L}$ 
on 
$L$ 
by 
\[
d_{L}(x, y)=
\begin{cases}
		d_{i}(x,y) & \text{if $x,y\in R_{i}$ for some $i$;}\\
		\epsilon\cdot \max\{2^{-i}, 2^{-j}\} & \text{if $x\in R_{i}, y\in R_{j}$ for some $i\neq j$; }\\
		\epsilon\cdot2^{-i} & \text{if $x=\infty, y\in R_{i}$ for some $i$;}\\
		\epsilon\cdot2^{-i} & \text{if $x\in R_{i}, y=\infty$ for some $i$.}
	\end{cases}
\]
On account of 
\ref{item:ep1017}, 
$d_{L}$ is actually a metric  (compare with \cite[Definition 3.3]{Ishiki2019}). 
We also observe that the space 
$(L, d_{L})$ 
is an 
$(\omega_0+1)$-metric space with
 $\delta_{d_{L}}(L)\le \epsilon$. 
By the properties 
\ref{item:rin1016} 
and 
\ref{item:rphi1016}
 of 
$\{(R_{i}, d_{i})\}_{i\in \nn}$, 
the metric space 
$(L, d_{L})$ 
satisfies the 
anti-$\mathfrak{G}$-transmissible property. 
\end{proof}
\begin{rmk}
It is also true that 
a transmissible parameter 
$\mathfrak{G}$ 
is singular if and only if 
there exists an 
$(\omega_0+1)$-metric space 
with arbitrary small diameter 
satisfying  the 
anti-$\mathfrak{G}$-transmissible property. 
\end{rmk}

Let 
$\mathfrak{G}$
 be a transmissible parameter. 
For a non-discrete metrizable space
 $X$, 
and for an
 $(\omega_0+1)$-subspace 
 $R$ 
 of
  $X$, 
we denote by 
$T(X, R, \mathfrak{G})$ 
the set of all 
$d\in \met{X}$
for which 
$(R, d|_{R^2})$ 
satisfies the 
anti-$\mathfrak{G}$-transmissible 
property.

As a consequence of 
Theorem \ref{thm:ext241009}, 
we obtain the following:
\begin{prop}\label{prop:dense}
Let 
$\mathfrak{G}=(Q, P, F, G, Z, \phi)$ 
be a singular transmissible parameter. 
Then for every non-discrete metrizable space
 $X$, 
and for every 
$(\omega_0+1)$-subspace 
$R$ 
of 
$X$, 
the set 
$T(X, R, \mathfrak{G})$ 
is dense in 
$\met{X}$.  
\end{prop}
\begin{proof}
Fix 
$d\in \met{X}$ 
and 
$\epsilon\in (0, \infty)$.  
From the singularity of 
$\mathfrak{G}$,  
by Lemma 
\ref{lem:sing}, 
it follows that there exists an 
$(\omega_0+1)$-metric space 
$(L, e)$
satisfying the 
anti-$\mathfrak{G}$-transmissible  property 
and  
$\delta_{e}(L)<\epsilon/2$. 
Since 
$R$ 
is an 
$(\omega_0+1)$-space, 
there exists an 
$(\omega_0+1)$ 
subspace 
$S$ 
of 
$R$ 
with 
$\delta_{d}(S)<\epsilon/2$. 
Let 
$\tau\colon  S\to L$ 
be a homeomorphism. 
By the definitions of 
$S$ 
and 
$e$, 
we have 
$\metdis_{S}(d|_{S^2}, \tau^{*}e)<\epsilon$. 
Theorem 
\ref{thm:ext241009}
guarantees 
the existence of  a metric 
$m\in \met{X}$ 
such that 
$m|_{S^2}=\tau^{*}e$ 
and 
$\metdis_{X}(m, d)<\epsilon$. 
Due to  Corollary 
\ref{cor:anti}, 
the metric space  
$(R, m|_{R^{2}})$ 
satisfies the 
anti-$\mathfrak{G}$-transmissible property. 
Since 
$d$ 
and 
$\epsilon$
 are arbitrary, 
we conclude that  
$T(X, R, \mathfrak{G})$
 is dense in 
$\met{X}$. 
\end{proof}

\begin{proof}[Proof of Theorem \ref{thm:trans}]
Let 
$X$ 
be a non-discrete metrizable space,  
and 
let 
$\mathfrak{G}$ 
be a singular transmissible parameter. 
Since 
$X$ 
is non-discrete, 
there exists an
 $(\omega_0+1)$-subspace 
$R$ 
of 
$X$. 
By the definitions, we have 
	\[
	T(X, R, \mathfrak{G})\subset S(X, \mathfrak{G}). 
	\] 
By Proposition
 \ref{prop:dense} 
 and 
 Corollary \ref{cor:open}, 
 the set 
 $S(X, \mathfrak{G})$ 
 is dense 
 $G_{\delta}$ 
 in 
 $\met{X}$. 
 This finishes the proof. 
\end{proof}

For a complete metrizable space 
$X$, 
we denote by 
$\cm(X)$ the set of 
all complete metrics in 
$\met{X}$. 
From  the latter part of
Theorem 
\ref{thm:ext241009}, 
we deduce the following:
\begin{thm}\label{thm:cm1}
Let 
$\mathfrak{G}$ 
be a singular transmissible parameter. 
For every non-discrete completely metrizable space $X$, 
the set of all 
$d\in \cm(X)$ 
for which 
$(X, d)$ 
satisfies 
the anti-$\mathfrak{G}$-transmissible property 
is dense 
$G_{\delta}$ 
in 
$\cm(X)$.  
\end{thm}

\subsection{The doubling property and the uniform disconnectedness}
For a metic space 
$(X, d)$ 
and for a subset 
$A$ 
of 
$X$, 
we set  
	\[
	\alpha_{d}(A)=\inf\{\, d(x, y)\mid \text{$x, y\in A$ and $x\neq y$}\, \}.
	\] 
A metric space 
$(X, d)$ 
is said to be 
\emph{doubling} 
if there exist 
$C\in (0, \infty)$ 
and 
$\lambda\in (0, \infty)$ 
such that  for every finite subset 
$A$ 
of 
$X$  
we have 
	\[
	\card(A)\le C\left(\frac{\delta_{d}(A)}{\alpha_{d}(A)}\right)^{\lambda}. 
	\]
Note that 
$(X, d)$ 
is doubling if and only if 
$(X, d)$ 
has finite Assouad dimension 
(see e.g., 
 \cite[Section 10]{MR1800917}). 

By the definitions of the topology of 
$\met{X}$,  
$\alpha_{d}$ 
and 
$\delta_{d}$,  
we obtain:
\begin{lem}\label{lem:bdcon}
Let 
$X$ 
be a metrizable space. 
 Fix  a finite subset
  $A$ 
  of 
  $X$.  
Then maps 
$\alpha_{*, A}, \delta_{*, A}\colon \met{X}\to \rr$
 defined by 
$\alpha_{*, A}(d)=\alpha_{d}(A)$ 
and 
$\delta_{*, A}(d)=\delta_{d}(A)$ 
is continuous. 
\end{lem}
\begin{prop}\label{prop:doubling:241018}
The doubling property on metric spaces is a transmissible  property with a singular transmissible parameter. 
\end{prop}
\begin{proof}
Define a map 
$F_{D}\colon  (\qq_{>0})^2\to \mathcal{F}((\rr_{> 0})^{2})$
by 
\[
	F_{D}((q_1, q_2))
	=
	\{\, (x, y)\in (\rr_{> 0})^2\mid x\le q_{1}\cdot y^{q_{2}}\, \}, 
\]
and define a constant map 
$G_{D}\colon  (\qq_{>0})^{2}\to \mathcal{P}(\nn)^{*}$
by 
$G_{D}(q)=[2, \infty)$.
Put 
$Z_{D}=\{1\}$. 
For each metrizable space 
$X$, 
and for each 
$q\in (\qq_{>0})^2$, 
define a map 
$\phi_{D}^{q, X}\colon  
\seq(G(q), X)\times Z_{D}\times \met{X}\to (\rr_{> 0})^2$
by 
	\[
	\phi_{D}^{q, X}(\{a_{i}\}_{i=1}^{N}, 1, d)=
	\left(
	N, 
	\frac{\delta_{d}(\{\, a_{i}\mid i\in \{1, \dots, N\}\, \})}{\alpha_{d}(\{\, a_{i}\mid i\in \{1, \dots, N\}\, \})}
	\right). 
	\]
Let 
$\mathfrak{DB}=((\qq_{>0})^2, (\rr_{>0})^{2}, F_{D}, G_{D}, \{1\}, \phi_{D})$. 
Then 
$\mathfrak{DB}$ 
satisfies 
the condition
\ref{item:tp2}
in Definition 
\ref{def:transp}. 
By the Lemma
 \ref{lem:bdcon}, 
we see that
 $\mathfrak{DB}$
  satisfies the condition 
  \ref{item:tp1}. 
Hence
 $\mathfrak{DB}$
  is a transmissible  parameter. 
The
 $\mathfrak{DB}$-transmissible  property
  is equivalent to 
the doubling property. 
We next prove that 
$\mathfrak{DB}$ 
is singular. 
For  
$q=(q_{1}, q_{2})\in (\qq_{>0})^{2}$ 
and for 
$\epsilon\in (0, \infty)$, 
we denote by  
$(R_{q}, d_{q})$ 
a finite metric space with 
$\card(R_q)>q_{1}+1$
and 
$d_{q}(x, y)=\epsilon$
whenever $x\neq y$.  
Then 
$\delta_{d_q}(R_q)=\epsilon$,  
and 
	\[
	\phi^{q, R_q}(R_q, 1, d_q)=(\card(R_q), 1)\not\in F_{D}(q).
	\] 
This implies the proposition. 
\end{proof}
\begin{rmk}
Let 
$\zz$ 
be the set of all integers with discrete topology, 
and let 
$h\in \met{\zz}$ be 
the relative metric on 
$\zz$ 
induced from the Euclidean metric on 
 $\rr$. 
Then
 $h$ 
 has a neighborhood
  $U$ 
  in 
$\met{X}$ 
such that for every 
$d\in U$ 
the space  
$(\zz, d)$
is doubling. 
\end{rmk}

A metric space 
$(X, d)$ 
is said to be 
\emph{uniformly disconnected} 
if there exists  
$\delta\in (0, 1)$
such that if a finite sequence 
$\{z_{i}\}_{i=1}^{N}$ 
in 
$X$ 
satisfies 
$\max_{i}d(z_{i}, z_{i+1})< \delta d(z_{1}, z_{N})$, 
then we have 
$N=1$. 
Namely, the inequality  $N>1$
implies  $\max_{i}d(z_{i}, z_{i+1})\ge \delta d(z_{1}, z_{N})$. 
Note that a metric space is uniformly disconnected 
if and only if 
it is bi-Lipschitz to an ultrametric space 
(see e.g.,  \cite[Lemma 5.1.10]{MR2662522}).

By the definition of the topology of 
$\met{X}$, 
we obtain:
\begin{lem}\label{lem:dcon}
Let 
$X$ 
be a metrizable space. 
Fix two points 
$a$, $b$ 
in 
$X$. 
Then a map 
$f\colon \met{X}\to \rr$ 
defined by 
$f(d)=d(a, b)$ 
is 
continuous.  
\end{lem}
\begin{prop}\label{prop:ud:241018}
The uniform disconnectedness on metic spaces 
is a transmissible  property with a singular parameter. 
\end{prop}
\begin{proof}
Define a map 
$F_{UD}\colon  \qq\cap (0, 1)\to \mathcal{F}((\rr_{\ge 0})^{2})$
by 
	\[
	F_{UD}(q)=\{\, (x, y)\in (\rr_{\ge 0})^{2}\mid x\ge qy\, \}, 
	\]
and define a constant map 
$G_{UD}\colon  \qq\cap (0, 1)\to \mathcal{P}^{*}(\nn)$
 by 
 $G_{UD}(q)=[2, \infty)$. 
Put 
$Z_{UD}=\{1\}$. 
 For each metrizable space 
 $X$,  
 and for each 
 $q\in \qq\cap (0, 1)$, 
define a map
$\phi_{UD}^{q, X}\colon  \seq(G_{UD}(q), X)\times Z_{UD}\times \met{X}\to (\rr_{\ge 0})^{2}$
by 
\[
	\phi_{UD}^{q, X}(\{a_{i}\}_{i=1}^{N}, 1, d)
	=
	\left(\max_{1\le i\le N-1}d(a_{i}, a_{i+1}), d(a_{1}, a_{N})\right). 
\]
Let 
$\mathfrak{UD}=(\qq\cap (0, 1), (\rr_{>0})^{2}, F_{UD}, G_{UD},  \{1\}, \phi_{UD})$.  
Then 
$\mathfrak{UD}$ 
satisfies the conditions 
\ref{item:tp2} in Definition \ref{def:transp}. 
By Lemma
 \ref{lem:dcon}, 
we see that
 $\mathfrak{UD}$ 
 satisfies the condition 
 \ref{item:tp1}. 
Hence 
$\mathfrak{UD}$  
is a transmissible parameter,  
and the 
 $\mathfrak{UD}$-transmissible property 
 is equivalent to 
the uniform disconnectedness. 
We next prove that 
$\mathfrak{UD}$
 is singular. 
For every 
$q\in \qq\cap (0,  1)$, 
take 
$n\in \nn$ 
with 
$1/n< q$. 
Put 
\[
	R_q=\{\, \epsilon\cdot i/n\mid i\in \zz\cap [0, n]\, \}, 
\] 
and let 
$d_{q}$
 be the relative metric on 
$R_{q}$ 
induced from the Euclidean metric. 
Then 
$\delta_{d_{q}}(R_{q})=\epsilon$,  
and 
\[
	\phi_{UD}^{q, R_q}(\{a_{i}\}_{i=1}^N, 1, d_q)=
	(\epsilon/n, \epsilon)\not\in U(q).
\] 
This leads to the proposition. 
\end{proof}
\begin{rmk}
Let 
$C$ 
be a countable discrete space, 
and let
 $h\in \met{C}$ 
 be a metric 
 such that 
 $h(x, y)=1$ whenever 
 $x\neq y$. 
Then 
$h$ 
has a neighborhood 
$U$ 
in 
$\met{C}$ 
such that for every 
$d\in U$ 
the space 
$(C, d)$ 
is uniformly disconnected. 
\end{rmk}

From 
Propositions 
\ref{prop:doubling:241018}
and 
\ref{prop:ud:241018} and from 
Theorem
 \ref{thm:trans},  
we deduce the next observations on 
dense 
$G_{\delta}$
subsets of moduli spaces of metrics. 
\begin{thm}\label{thm:dandud:241018}
For every
non-discrete 
 metrizable space 
$X$,  
the following two sets are 
dense $G_{\delta}$ in 
$\met{X}$. 
\begin{enumerate}
\item 
the set of all metrics 
$d\in \met{X}$ 
for which 
$(X, d)$ is not doubling; 
\item
the set of all metrics 
$d\in \met{X}$ 
for which 
$(X, d)$ is not 
uniformly disconnected. 
\end{enumerate}
\end{thm}

\subsection{Rich pseudo-cones}
Let 
$(X, d)$
 be a metric space. 
Let 
$\{A_{i}\}_{i\in \nn}$ 
be a sequence of subsets of 
$X$,  
and let  
$\{u_{i}\}_{i\in \nn}$ 
be a sequence in 
$(0, \infty)$. 
We say that a metric space 
$(P, d_{P})$ 
is a 
\emph{pseudo-cone of $X$} 
approximated by 
$(\{A_{i}\}_{i\in \nn}, \{u_{i}\}_{i\in \nn})$ 
if 
\[
	\lim_{i\to \infty}\mathcal{GH}((A_{i}, u_{i}\cdot d|_{A_{i}^2}), (P, d_{P}))=0
\]
(see \cite{MR4173165}), 
where 
$\mathcal{GH}$ 
is the Gromov--Hausdorff distance (see \cite{MR1835418}). 
For  a metric space 
$(X, d)$,
 we denote by 
 $\pc(X, d)$ 
 the class of all pseudo-cones of 
 $(X, d)$.
Let 
$\mathscr{F}$ 
be the class of 
all finite metric spaces whose  distances 
are rational numbers. 
We denote by 
$\mathscr{G}$ 
the quotient class of 
$\mathscr{F}$ 
divided by the isometric equivalence. 
Note that 
$\mathscr{G}$ is countable.

We say that a metric space 
$(X, d)$ 
\emph{has rich pseudo-cones} 
if 
$\mathscr{F}$
 is contained in 
 $\pc(X, d)$. 
\begin{prop}
The rich pseudo-cones property on metric spaces  is 
an anti-transmissible  property with a singular transmissible  parameter. 
\end{prop}
\begin{proof}
Let 
$\{(L_{n}, d_{n})\}_{n\in \nn}$ 
be a complete representation system of 
$\mathscr{G}$. 
Let 
$L_{n}=\{f_{n, l}\}_{l=1}^{\card(F_{i})}$. 
Define a function 
$F_{R}\colon  \nn^2\to \mathcal{F}(\rr)$ 
by 
\[
	F_{R}(n, m)=\{\, y\in \rr\mid y\ge 2^{-m}\, \}, 
\]
and define a map 
$G_{R}\colon \nn^{2}\to \mathcal{P}^{*}(\nn)$ 
by 
$G_{R}(n, m)=\{\card(L_{n})\}$.

For each  
$k=(n, m)\in \nn^{2}$,  
for each metrizable space 
$X$,  
for each 
$\{a_{i}\}_{i=1}^{M}\in \seq(G_{R}(k), X)$,  
and for all 
$i, j\in \{1, \dots, M\}$ 
we define a function 
$r_{i, j}^{k}(\{a_i\}_{i=1}^{M})
\colon  (0, \infty)\times \met{X}\to \rr$ 
by 
\begin{align*}
	r_{i, j}^{k}(\{a_{i}\}_{i=1}^{M})(z, d)
	=
	|z^{-1}d(a_{i}, a_{j})-d_{n}(f_{n, i}, f_{n, j})| 
\end{align*}
if 
$i, j\in \{1, \dots, M\}$; 
otherwise, 
we define 
$r_{i, j}^{k}(\{a_i\}_{i=1}^{M})(z, d)=0$.
By Lemma
 \ref{lem:dcon}, 
 the map 
 $r_{i, j}^k(\{a_i\}_{i=1}^{M})$
  is continuous.

Define a map 
$\phi_{R}^{k, X}\colon  \seq(G(k), X)\times(0, \infty)\times \met{X}\to \rr$ 
by 
\[
	\phi_R^{k, X}(\{a_{i}\}_{i=1}^{M}, z, d)
	=
	\max_{i, j\in \{1, \dots, M\}}r_{i, j}^{k}(\{a_{i}\}_{i=1}^{M})(z, d). 
\]
Let 
$\mathfrak{R}=(\nn^{2}, \rr, F_{R},  G_{R}, (0, \infty), \phi_r)$. 
Then 
$\mathfrak{R}$ 
satisfies the conditions 
\ref{item:tp1} and 
\ref{item:tp2}
in Definition 
\ref{def:transp}, 
and hence 
it is a transmissible parameter.

For a metric space 
$(X, d)$, 
the 
anti-$\mathfrak{R}$-transmissible property means that 
for every 
$n\in \nn$, 
and for every 
$m\in \nn$, 
there exist 
a finite subspace 
$A=\{a_{i}\}_{i=1}^{\card(L_{n})}$ 
of 
$X$ 
and a positive number 
$z\in (0, \infty)$ 
such that for all 
$i, j\in \{1, \dots, \card(L_{n})\}$ 
we have 
\[
	|z^{-1}d(a_{i}, a_{j})-d_{n}(f_{n,i}, f_{n, j})|<2^{-m};
\]
in particular, 
$\mathcal{GH}((A, z^{-1}d|_{A^{2}}), (L_{n}, d_{n}))<2^{-(m+1)}$. 
Thus 
$\mathscr{F}$ 
is contained in 
$\pc(X, d)$. 
This implies that 
$(X, d)$ 
has rich pseudo-cones. 
We next prove the opposite. 
If 
$(X, d)$ 
has rich pseudo-cones,  
then for every 
$(W, d_{W})\in \mathscr{F}$, 
and for every 
$\epsilon\in (0, \infty)$, 
there exist a positive number 
$z\in (0, \infty)$ 
and  a subset 
$A$ 
of
 $X$ 
 with 
$\card(A)=\card(W)$ such that 
$\mathcal{GH}((A, z^{-1}d|_{A^{2}}), (W, d_{W}))<\epsilon$. 
Thus 
$(X, d)$ 
satisfies the
 anti-$\mathfrak{R}$-transmissible property. 
We next prove that 
$\mathfrak{R}$ 
is singular. 
For each 
$(n, m)\in \nn^{2}$ 
and for each 
$\epsilon\in (0, \infty)$, 
we put 
\[
	(R, d_{R})=\left(L_{n}, \frac{\epsilon}{\delta_{d_{n}}(L_{n})}\cdot d_{n}\right).
\]
Then we have 
$\delta_{d_{R}}(R)=\epsilon$,  
and 
\[
	\phi_R^{(n, m), R}
	\left(
	\{f_{n, l}\}_{l=1}^{\card(L_{n})}, \delta_{d_{n}}(L_{n})/\epsilon, d_{R} 
	\right)
	=
	0
	\not\in F_{R}(n, m). 
\] 
Therefore 
$\mathfrak{R}$
 is singular. 
This completes the proof. 
\end{proof}

Since every compact metric space is arbitrarily approximated by 
members of 
$\mathscr{F}$
 in the sense of Gromov--Hausdorff, 
we obtain: 
\begin{prop}
A metric space 
$(X, d)$ 
has rich pseudo-cones 
if and only if 
$\pc(X, d)$ 
contains all compact metric spaces.  
\end{prop}

From Theorem
 \ref{thm:trans},  
we deduce the following:
\begin{thm}\label{thm:richcones}
For every
non-discrete 
 metrizable space 
$X$,  
the set of all metrics 
$d\in \met{X}$ 
for which 
$(X, d)$ 
has rich pseudo-cones is dense 
$G_{\delta}$ 
in 
$\met{X}$. 
\end{thm}

\begin{rmk}
Chen and Rossi 
\cite{MR3395963}
 introduced the notion of locally rich compact metric spaces. 
 They investigated the distribution of locally rich metric spaces in a space of  compact metric spaces with respect to the Gromov--Hausdorff distance, 
 and they also studied this subject in the Euclidean setting in a space of compact subspaces. 
\end{rmk}
\subsection{Metric inequality}

Let 
$f\colon  \rr^{\binom{n}{2}}\to \rr$ 
be a continuous function. 
We say that a metric space 
$(X, d)$ 
satisfies the 
\emph{$(n, f)$-metric inequality}
 if 
for all $n$ points 
$a_{1}, \dots, a_{n}$ 
in 
$X$ 
we have 
$f(\{d(a_{i}, a_{j})\}_{i\neq j})\ge 0$. 
We say that a function  
$f\colon \rr^{\binom{n}{2}}\to \rr$ 
is \emph{positively sub-homogeneous} 
if there exists 
$s\in (0, \infty)$ 
such that
for every 
$x\in \rr^{\binom{n}{2}}$ 
and for every 
$c\in (0, \infty)$
we have 
$f(r\cdot x)\le r^{c}f(x)$.

\begin{prop}\label{prop:inq}
For 
$n\in \nn$, 
let  
$f\colon  \rr^{\binom{n}{2}}\to \rr$ 
be a continuous function. 
Then satisfying the 
$(n, f)$-metric inequality 
on metric spaces is a transmissible  property. 
Moreover, 
if 
$f$
 is positively sub-homogeneous,  
and if  there exists a metric space 
that does not satisfy 
the 
$(n, f)$-metric inequality, 
then satisfying the 
$(n, f)$-metric inequality on metric spaces  
is 
a  transmissible  property 
with a singular transmissible parameter. 
\end{prop}
\begin{proof}
Let
 $Q=\{1\}$ 
and define a map  
$F\colon Q\to \mathcal{F}(\rr)$ 
by 
$F(1)=[0, \infty)$. 
Define a map 
$G\colon  Q\to \mathcal{P}^*(\nn)$ 
by 
$G(1)=\{n\}$. 
For each  metrizable space
 $X$,  
we define a map
$\phi^{1, X}\colon  
\seq(\{n\}, X)\times \{1\}\times \met{X}\to \rr$ 
by 
\[
	\phi^{1, X}(\{a_{i}\}_{i=1}^{n}, 1, d)=f(\{d(a_{i}, a_{j})\}_{i\neq j}). 
\]
Let 
$\mathfrak{G}=(\{1\}, \rr,  F, G, \{1\}, \phi)$. 
Then 
$\mathfrak{G}$ 
is a transmissible parameter.

We next show the latter part. 
Since there exists a metric space
 that does not satisfy 
  the 
$(n, f)$-metric inequality, 
there exists a metric space 
$(S, d_S)$ 
with 
$\card(S)=n$ 
that does not satisfy the 
$(n, f)$-metric inequality. 
Let 
$c\in (0, \infty)$ 
be a positive number such that for every 
$x\in \rr^{\binom{n}{2}}$, 
and for every 
$r\in (0, \infty)$ 
we have 
$f(r\cdot x)\le r^{c}f(x)$. 
Let 
$S=\{s_{i}\}_{i=1}^n$ 
and assume that 
$f(\{d_{S}(s_{i}, s_{j})\}_{i\neq j})<0$. 
For every 
$\epsilon\in (0, \infty)$,  
put $(R, d_{R})=(S, \epsilon\cdot d_{S})$. 
Thus we have 
$\delta_{d_{R}}(R)=\epsilon$,  
and 
\[
	\phi^{1, R}(\{s_{i}\}_{i=1}^{n}, 1, d_R)=f(\{\epsilon\cdot d_{S}(s_{i}, s_{j})\})
	=\epsilon^cf(\{d_{S}(s_{i}, s_{j})\})<0. 
\]
This implies that 
$\mathfrak{G}$ 
is singular. 
This finishes the proof. 
\end{proof}
Combining Theorem \ref{thm:trans},  
Corollary \ref{cor:open} 
and Proposition \ref{prop:inq}, 
we obtain the following corollary:
\begin{cor}\label{cor:inqopen}
Let 
$X$ 
be a non-discrete metrizable space. 
For 
$n\in \nn$, 
let 
$f\colon  \rr^{\binom{n}{2}}\to \rr$ 
be a continuous function. 
If
 $f$ 
 is positively sub-homogeneous,  
and if  there exists a metric space not satisfying the 
$(n, f)$-metric inequality, 
then the set of all metrics 
$d$
 in 
$\met{X}$ 
for which the space 
$(X, d)$
does not satisfy the 
$(n, f)$-metric inequality is 
dense open in 
$\met{X}$. 
\end{cor}

We  say that 
a metric space 
$(X, d)$ 
satisfies the 
\emph{ultratriangle inequality} 
if 
for all three points 
$a_{1}, a_{2}, a_{3}$ 
in 
$X$ 
we have 
\[
	d(a_{1}, a_{3})\le \max\{d(a_{1}, a_{2}), d(a_{2}, a_{3})\}. 
\]

\begin{prop}
Define a function 
$f\colon  \rr^{\binom{3}{2}}\to \rr$ 
by 
\[
	f(x)=\max\{x_{1, 2}, x_{2, 3}\}-x_{1, 3}. 
\]
Then 
the ultrametric  inequality on metric spaces is equivalent to the 
$(3, f)$-metric inequality,  
and 
$f$
 is positively sub-homogeneous. 
\end{prop}

We  say that a metric space  
$(X, d)$ 
satisfies the 
\emph{Ptolemy inequality} if 
for all four points 
$a_{1}, a_{2}, a_{3}, a_{4}$ 
in 
$X$ 
we have 
\[
	d(a_{1}, a_{3})d(a_{2}, a_{4})\le d(a_{1}, a_{2})d(a_{3}, a_{4})+d(a_{1}, a_{4})d(a_{2}, a_{3}). 
\]
\begin{prop}
Define a function  
$f\colon  \rr^{\binom{4}{2}}\to \rr$ 
by 
\[
	f(x)=x_{2, 3}x_{1, 4}+x_{1,2}x_{3, 4}-x_{1, 3}x_{2, 4}. 
\]
Then 
the Ptolemy inequality on metric spaces  is equivalent to the  
$(4, f)$-metric inequality, 
 and 
$f$ 
is positively sub-homogeneous. 
\end{prop}

Gromov
 \cite{MR1871000} 
introduced the cycle condition for metric spaces as follows:  
Let 
$m\in \nn$ 
and 
$\kappa\in \rr$. 
Let 
$(M(\kappa), d_{M(\kappa)})$ 
be the two-dimensional space form of 
constant  curvature 
$\kappa$. 
We say that a metric space 
$(X, d)$ 
satisfies the 
\emph{$\cycl_m(\kappa)$ condition}
if for every map 
$f\colon \zz/m\zz\to X$ 
there exists a map 
$g\colon  \zz/m\zz\to M(\kappa)$ 
such that 
\begin{enumerate}[label=\text{(\arabic*)}]
	\item\label{item:gcyc11018}
	for all 
	$i\in \zz/m\zz$, 
	we have 
\[
	d_{M(\kappa)}(g(i), g(i+1))\le d(f(i), f(i+1));
\] 
	\item\label{item:gcyc21018}
	 for all 
	$i, j\in \zz/m\zz$ 
	with 
	$i-j\neq \pm 1$, 
	we have 
\[
	d_{M(\kappa)}(g(i), g(j))\ge d(f(i), f(j)), 
\]
where the symbol 
$``+''$ 
stands for the addition of 
$\zz/m\zz$. 
\end{enumerate}
\begin{prop}
For every  
$m\in \nn$, 
the 
$\cycl_{m}(0)$  
condition can be represented by an 
$(m, C)$-metric inequality 
for some positively sub-homogeneous function
 $C$.  
\end{prop}
\begin{proof}
For a map 
$g\colon  \zz/m\zz\to \rr^2$, 
we define two functions 
$C_{1, g}, C_{2, g}:\rr^{\binom{m}{2}}\to \rr$ 
by 
\begin{align*}
	&C_{1, g}(x)=
	\min_{i\in \zz/m\zz}\{x_{i, i+1}-d_{M(0)}(g(i), g(i+1))\}, \\
	&C_{2, g}(x)=
	\min_{i, j\in\zz/m\zz,\  i-j\neq \pm1 }\{d_{M(0)}(g(i), g(j))-x_{i, j}\}. 
\end{align*}
We define a function  
$C\colon \rr^{\binom{m}{2}}\to \rr$ 
by 
\[
C(x)=
	\sup_{g:\zz/m\zz\to M(0)}
	\{
	C_{1, g}(x), \ 
	C_{2,  g}(x)
	\}.
\]
Then 
$C$ 
is continuous. 
For every 
$r\in (0, \infty)$ 
we have 
\begin{align*}
C(r\cdot x)&=\sup_{g:\zz/m\zz\to M(0)}
	\{
	C_{1, g}(r\cdot x), \ 
	C_{2,  g}(r\cdot x)
	\}\\
	&=
	r\cdot \sup_{g:\zz/m\zz\to M(0)}
	\{
	C_{1, g/r}(r\cdot x), \ 
	C_{2,  g/r}(r\cdot x)
	\}\\
	&=
	r\cdot \sup_{g:\zz/m\zz\to M(0)}
	\{
	C_{1, g}( x), \ 
	C_{2 g}(x)
	\}. 
\end{align*}
Thus the function 
$C$ 
is  positively sub-homogeneous.

If 
 $m$-many 
  points 
$a_{1}, \dots, a_{m}$ 
in 
$X$ 
satisfy 
 $C(\{d(a_i, a_j)\}_{i\neq j})\ge 0$, 
 then 
there exists a map 
$g\colon  \zz/m\zz\to M(0)$ 
such that 
$C_{1, g}(\{d(a_i, a_j)\}_{i \neq j})\ge 0$ 
and 
$C_{2, g}(\{d(a_i, a_j)\}_{i \neq j})\ge 0$. 
These two inequalities are equivalent to the conditions 
\ref{item:gcyc11018}
and 
\ref{item:gcyc21018} 
in the $\cycl_m(0)$
 condition, respectively. 
Therefore the 
$\cycl_m(0)$
 condition is equivalent to the 
 $(m, C)$-metric inequality. 
\end{proof}

\subsection{Gromov hyperbolicity}

Gromov 
\cite{MR919829}
 introduced the notion of 
the Gromov hyperbolicity. 
We say that 
$(X, d)$ 
is \emph{Gromov hyperbolic} if 
there exists 
$\delta\in [0, \infty)$ 
such that for all four points 
$a_{1}, a_{2}, a_{3}, a_{4}$ 
in  
$X$ 
we have 
\begin{multline*}
	d(a_{1}, a_{3})+d(a_{2}, a_{4})\\
	\le \max\{d(a_{1}, a_{2})+d(a_{3}, a_{4}), d(a_{1}, a_{4})+d(a_{2}, a_{3})\}+2\delta, 
\end{multline*}

\begin{prop}
 The Gromov hyperbolicity on metric spaces  is equivalent to a transmissible parameter. 
\end{prop}
\begin{proof}
Let
 $Q=\qq_{\ge 0}$. 
 For each 
 $\delta\in Q$, 
define a function 
$g_{\delta}\colon  \rr^{\binom{4}{2}}\to \rr$ 
by 
\[
	g_{\delta}(x)=
\max\{x_{1, 2}+x_{3, 4},  x_{1, 4}+x_{2, 3}\}+2\delta-(x_{1, 3}+x_{2, 4}). 
\]
We also  define a map  
$F\colon Q\to \mathcal{F}(\rr)$ 
by 
$F(\delta)=[0, \infty)$. 
Define a map 
$G\colon  Q\to \mathcal{P}^*(\nn)$ 
by 
$G(\delta)=\{4\}$. 
For each  metrizable space
 $X$,  
we define a map
$\phi^{\delta, X}\colon  
\seq(\{4\}, X)\times \{1\}\times \met{X}\to \rr$ 
by 
\[
	\phi^{\delta, X}(\{a_{i}\}_{i=1}^{n}, 1, d)=g_{\delta}(\{d(a_{i}, a_{j})\}_{i\neq j}). 
\]
Let 
$\mathfrak{G}=(Q, \rr,  F, G, \{1\}, \phi)$. 
Then 
$\mathfrak{G}$ 
is a transmissible parameter. 
Under this notations, 
we observe that 
$(X, d)$ is Gromov hyperbolic 
if and only if 
it satisfies the 
$\mathfrak{G}$-transmissible property. 
\end{proof}

Since for every metrizable space 
$X$ 
the set of all bounded metrics in 
$\met{X}$ 
is open in
$\met{X}$,  
and since every bounded metric space is Gromov hyperbolic, 
we obtain the following:
\begin{prop}\label{prop:Gro}
The Gromov hyperbolicity on metric spaces  is 
not equivalent to any transmissible  property 
with a singular transmissible parameter. 
\end{prop}

\section{Local Transmissible  properties}\label{sec:loctrans}
In this section, 
we prove Theorem 
\ref{thm:loctrans}. 
The following lemma plays  a key role in the proof of Theorem \ref{thm:loctrans}. 
\begin{lem}\label{lem:Baire}
For 
every  second-countable  locally compact metrizable space 
$X$, 
the space 
$\met{X}$
 is  completely metrizable. 
 In particular,  it is Baire. 
\end{lem}
\begin{proof}
Let 
$C(X^{2})$ 
be the set of all real-valued continuous functions on $X^{2}$.
We define a metric 
$\mathcal{E}$ 
on
 $C(X^{2})$ 
by 
\begin{equation*}\label{eq:E}
\mathcal{E}(f, g)=\min\left\{1,\sup_{(x, y)\in X^{2}}|f(x, y)-g(x, y)|\right\}. 
\end{equation*}
Note that
the metric 
$\mathcal{E}|_{\met{X}^{2}}$ 
on 
$\met{X}$  
generates the same topology as  
$\metdis_{X}$ 
on 
$\met{X}$, 
which coincides  with the uniform topology.

Note that  the space 
$(C(X^{2}), \mathcal{E})$ 
is completely metrizable. 
By Lemma 
\ref{lem:gdelta}, 
to prove the lemma, 
it suffices to show that 
$\met{X}$ 
is 
$G_{\delta}$ 
in  
$(C(X^{2}), \mathcal{E})$.
We denote by
 $\yopset$ 
 the set of all 
$f\in C(X^{2})$ 
such that 
\begin{enumerate}
	
	\item for every 
	$x\in X$ 
	we have 
	$f(x, x)=0$;
	
	\item 
	for all
	 $x, y\in X$ 
	 we have 
	 $f(x, y)=f(y, x)$
	 and 
	 $f(x, y)\ge 0$;
	 
	\item for every triple 
	$x, y, z\in X$
	 we have  $
	 f(x, y)\le f(x, z)+f(z, y)$. 
\end{enumerate}
Namely, 
$\yopset$ 
is the set of all 
continuous pseudo-metrics on 
$X$. 
Note that 
$\yopset$ 
is 
a closed subset in  the metric space 
$(C(X^{2}), \mathcal{E})$. 
Since all closed subsets 
of a metric space are 
$G_{\delta}$ 
in the whole space, 
the set 
$\yopset$ 
is 
$G_{\delta}$ 
in 
$(C(X^{2}), \mathcal{E})$.

Using the assumption that 
$X$ 
is second-countable and  locally compact, 
now  we take a sequence 
$\{\yodset{n}\}_{n\in \nn}$ 
of compact subsets of 
$X^{2}$ 
with 
$\bigcup_{n\in \nn}\yodset{n}=X^{2}\setminus \Delta_{X}$, 
where 
$\Delta_{X}$ 
is the diagonal set of 
$X^{2}$, 
and take a sequence 
$\{\yokset{n}\}_{n\in \nn}$ 
of compact subsets of 
$X$ 
with 
$\yokset{n}\subset \nai(\yokset{n+1})$ 
and 
$\bigcup_{n\in \nn} \yokset{n}=X$, 
where 
``$\nai$''
means the interior.

For every 
$n\in \nn$,  
let
 $\yolset{n}$ 
be the set of all 
$f\in C(X^{2})$ 
for which 
there exists
$c\in (0, \infty)$ 
such that 
for each  
$s>n$ 
we have 
\[
\inf\{\, f(x, y)\mid x\in \yokset{n}, \  y\in X\setminus \yokset{s}\, \} >c. 
\]

For each 
$n\in \nn$,  
let 
$\yoeset{n}$ 
be 
the set of all 
$f\in C(X^{2})$ 
such that
for each 
$(x, y)\in \yodset{n}$ 
we have 
$0<f(x, y)$.  
Note that each 
$\yolset{n}$ 
and each 
$\yoeset{n}$ 
are open subsets 
of 
$(C(X^{2}), \mathcal{E})$. 

We next prove that
\[
	\met{X}=
	\yopset
	\cap 
	\left(
	\bigcap_{n\in \nn} \yolset{n}
	\right)
	\cap  
	\left(
	\bigcap_{n\in \nn}\yoeset{n}
	\right). 
\]
To prove
$\met{X}\subset
	\yopset
	\cap 
	\left(
	\bigcap_{n\in \nn} \yolset{n}
	\right)
	\cap  
	\left(
	\bigcap_{n\in \nn}\yoeset{n}
	\right)$, 
take 
$d\in \met{X}$. 
Since 
$d$
 is a metric on 
 $X$, 
we have
 $d\in \yopset$ 
and  
$d\in \bigcap_{n\in \nn}\yoeset{n}$. 
To show 
$d\in \bigcap_{n\in \nn}\yolset{n}$,  
take an arbitrary number 
$n\in \nn$. 
Since 
$\yokset{n}\subset \nai(\yokset{n+1})$ 
and 
since 
$d\in \met{X}$, 
for every 
$s> n$ 
we have 
\begin{align*}
	0&<d(\yokset{n}, X\setminus \nai(\yokset{n+1}))
	\le
	d(\yokset{n}, X\setminus \yokset{n+1})
	\\
	&\le
	d(\yokset{n}, X\setminus \yokset{s}).
\end{align*}
Thus  
$d\in \bigcap_{n\in \nn}\yolset{n}$. 
Hence we obtain
\[
	\met{X}
	\subset 
	\yopset
	\cap 
	\left(
	\bigcap_{n\in \nn} \yolset{n}
	\right)
	\cap
	\left(
	\bigcap_{n\in \nn}\yoeset{n}
	\right). 
\]
Next we  take 
$d\in \yopset\cap \left(\bigcap_{n\in \nn} \yolset{n}\right)\cap  \left(\bigcap_{n\in \nn}\yoeset{n}\right)$. 
Since 
$d\in \yopset\cap \left( \bigcap_{n\in \nn}\yoeset{n}\right)$, 
the function 
$d$ 
is  continuous on 
$X^{2}$ 
and it is  a metric on 
$X$. 
Fix 
$e\in \met{X}$. 
We show that  
the metric  
$d$ 
is topologically equivalent to
 $e$. 
Since 
$d$ 
is continuous on 
 $X^{2}$, 
the metric
 $d$ 
 generates a weaker topology than that of
  $(X, e)$. 
Namely, 
if 
$\lim_{n\to \infty}x_{n}= a$ 
in 
$(X, e)$, 
then 
$\lim_{n\to \infty}x_{n}=a$ 
in 
$(X, d)$. 
Assume next  that 
$\lim_{n\to \infty}x_{n}= a$
 in 
 $(X, d)$. 
Our purpose is 
to prove 
that 
$\lim_{n\to \infty}x_{n}\to a$
in 
$(X, e)$. 
Since 
$\{\yokset{n}\}_{n\in \nn}$ 
is a covering of 
$X$, 
there exists 
$s\in \nn$
such that 
$a\in \yokset{s}$. 
For the sake of contradiction, 
suppose that 
there exist infinitely many 
$i\in \nn$ 
with
 $(X\setminus \yokset{i})\cap \{\, x_{n}\mid {n\in \nn}\, \}\neq \emptyset$. 
Then we have
\[
	\liminf_{i\to \infty}d(\yokset{s}, X\setminus \yokset{i})
	\le 
	\lim_{j\to \infty}d(a, x_{j})
	=
	0.
\] 
This contradicts the fact that 
$d\in \bigcap_{n\in \nn}\yolset{n}$. 
Hence there exists 
a sufficiently large number 
$m\in \nn$ 
such that  
the sequence 
$\{x_{n}\}_{n\in \nn}$ 
is contained in 
$\yokset{m}$. 
 We may assume that 
 $s\le m$, which implies that 
 $a\in \yokset{m}$. 
By the compactness of 
$\yokset{m}$ and by 
the fact that 
$(X, d)$ 
is Hausdorff, 
the map 
$1_{\yokset{m}}\colon (\yokset{m}, e)\to (\yokset{m}, d)$
becomes a homeomorphism. 
Thus 
the restricted metrics 
$e|_{\yokset{m}^{2}}$ and 
$d|_{\yokset{m}^{2}}$ 
generate the same topology on 
$\yokset{m}$. 
Since  
$\lim_{n\to\infty}x_{n}=a$ in $(X, d)$
and since 
$\{x_{n}\}_{n\in \nn}$
 is 
 contained in 
$\yokset{m}$, 
we see that 
$\lim_{n\to\infty}x_{n}=a$
in 
$(X, e)$. 

As a result, 
the metric 
$d$
 generates the same topology as 
 $e$,  
and hence 
\[
	\met{X}
	\supset 
	\yopset
	\cap 
	\left(
	\bigcap_{n\in \nn} \yolset{n}
	\right)
	\cap  
	\left(
	\bigcap_{n\in \nn}\yoeset{n}
	\right). 
\]
Therefore we conclude that 
$\met{X}$ 
is a 
$G_{\delta}$ 
subset of 
$(C(X^{2}), \mathcal{E})$. 
\end{proof}

\begin{proof}[Proof of Theorem \ref{thm:loctrans}]
Let 
$X$ 
be a second-countable,  
locally compact locally non-discrete space, 
and let 
$\mathfrak{G}$ 
be a  singular transmissible parameter. 
Let 
$\mathfrak{G}=(Q, P, F, G, Z, \phi)$, 
and 
let 
$S$ 
be the set of all metrics 
$d\in \met{X}$ 
for which 
$(X, d)$ 
satisfies 
the local 
anti-$\mathfrak{G}$-transmissible property. 
Take
 a countable open base 
 $\{U_{i}\}_{i\in \nn}$ 
 of $X$,  and take 
 a family 
 $\{R_{i}\}_{i\in \nn}$
 of 
$(\omega_0+1)$-subspaces of 
$X$ with 
$R_{i}\subset U_i$. 
Since 
$\{U_{i}\}_{i\in \nn}$ 
is  an open base of 
$X$, 
according to 
 Lemma
 \ref{lem:here}, 
we have
\[
	S
	=
	\bigcap_{i\in \nn}
	\bigcap_{q\in Q}
	\bigcup_{z\in Z}
	\bigcup_{a\in \seq(G(q), U_{i})}
	S(X, \mathfrak{G}, q, a, z). 
\]
Corollary  \ref{cor:open} 
implies that 
$S$ 
is 
$G_{\delta}$ 
in 
$\met{X}$. 
For every 
$i\in \nn$,  
the set 
\[
	\bigcap_{q\in Q}
	\bigcup_{z\in Z}
	\bigcup_{a\in \seq(G(q), U_i)}
	S(X, \mathfrak{G}, q, a, z)
\]
contains 
$T(X, R_i, \mathfrak{G})$.  
Proposition \ref{prop:dense} 
implies that 
each 
$T(X, R_{i}, \mathfrak{G})$ 
is dense in 
$\met{X}$. 
By Lemma \ref{lem:Baire},  
the space 
$\met{X}$ 
is Baire, 
and hence 
$S$ 
is dense 
$G_{\delta}$ 
in 
$\met{X}$. 
This completes the proof.
\end{proof}




\bibliographystyle{myplaindoidoi}
\bibliography{bibtexmet/bibmet.bib}

\end{document}